\documentclass[letterpaper]{amsart}
\pdfoutput=1
\usepackage{amsmath}
\usepackage{amsfonts}
\usepackage{amssymb}
\usepackage{amsthm}
\usepackage{graphicx}
\usepackage{bbm}
\usepackage{hyperref}
\usepackage{cite}

\def\thetitle{A surprising fibration of \texorpdfstring{$S^3 \times S^3$}{S3 x S3} by great 3-spheres}
\def\theauthor{Haggai Nuchi}

\hypersetup{pdftitle=\thetitle, pdfauthor=\theauthor, colorlinks=true, linkcolor=black, citecolor=black, filecolor=black, urlcolor=black}

\newcommand{\R}{\mathbb{R}}
\newcommand{\Z}{\mathbb{Z}}
\newcommand{\HH}{\mathbb{H}}
\newcommand{\F}{\mathcal{F}}

\newtheorem{thm}{Theorem}[section]
\newtheorem{prop}[thm]{Proposition}
\newtheorem{lem}[thm]{Lemma}
\newtheorem{cor}[thm]{Corollary}

\theoremstyle{definition}
\newtheorem{dfn}[thm]{Definition}

\theoremstyle{remark}
\newtheorem{rmk}[thm]{Remark}

\title[\thetitle]{\LARGE \thetitle}
\author{\theauthor}
\date{}

\begin{document}
  \begin{abstract}
    In this paper, we describe a new surprising example of a fibration of the Clifford torus $S^3\times S^3$ in the 7-sphere by great 3-spheres, which is fiberwise homogeneous but whose fibers are not parallel to one another. In particular it is not part of a Hopf fibration.
    
    A fibration is fiberwise homogeneous when for any two fibers there is an isometry of the total space taking fibers to fibers and taking the first given fiber to the second one.
    
    We also describe in detail the geometry of this surprising fibration and how it differs from the Hopf fibration.
  \end{abstract}
  \maketitle

  \section{Introduction}\label{Sec:Intro}
    \subsection{Background and definitions}
      The Hopf fibrations of $S^{2n+1}$ by great circles, $S^{4n+3}$ by great 3-spheres, and $S^{15}$ by great 7-spheres are the prototypical examples of fibrations of spheres by subspheres \cite{hopf1931abbildungen, hopf1935abbildungen}.

      They have many beautiful properties. They are unique among fibrations of spheres by great subspheres in having parallel fibers \cite{gromoll1988low, wilking2001index}. We show in another paper \cite{nuchi2014hopf} that they are also characterized among subsphere fibrations by having a symmetry group large enough to take any fiber to any other fiber.
      \begin{dfn}
        Let $\F$ be a fibration of a Riemannian manifold $(M,g)$. We say that $\F$ is {\em fiberwise homogeneous} if for any two fibers there is an isometry of $(M,g)$ taking fibers to fibers and taking the first given fiber to the second one. See Figure~\ref{Fig:ExampleFWH} for an example.
      \end{dfn}
      \begin{figure}[h]
        \centering
        \includegraphics[width=0.3\textwidth]{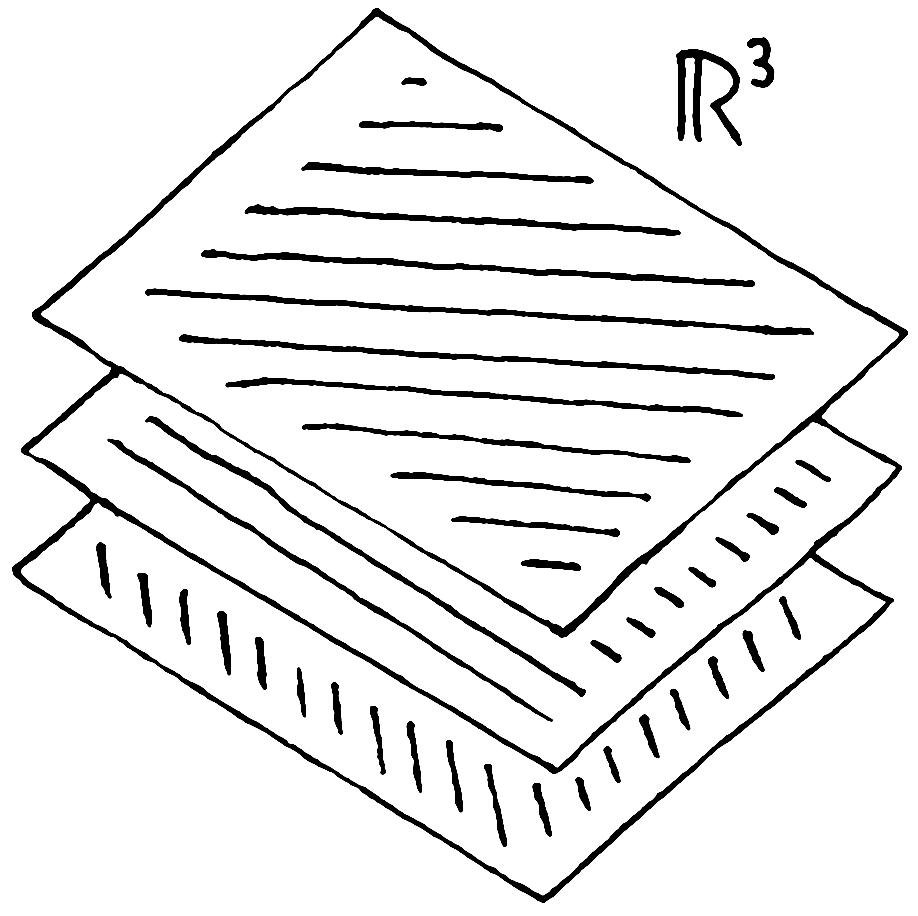}
        \caption{An example of a fiberwise homogeneous fibration of Euclidean 3-space by straight lines. Layer 3-space by parallel planes, and fill each plane by parallel lines, with the angle changing at a constant rate with respect to the height of the plane. The fibers are not parallel. See \cite{nuchi2014space}.}
        \label{Fig:ExampleFWH}
      \end{figure}
    
      The main theorem of \cite{nuchi2014hopf}, a companion paper to this one, is that the Hopf fibrations are characterized among fibrations of spheres by subspheres by being fiberwise homogeneous.
    
      In this paper we focus on the unit 7-sphere, embedded in $\R^8 = \R^4 \times \R^4$. Consider the ``Clifford Torus'' $S^3 \times S^3$ in the 7-sphere, where each factor is the sphere of radius $1/\sqrt{2}$ in its respective $\R^4$ factor. The main result here is that there are fibrations of this Clifford Torus by great 3-spheres which are fiberwise homogeneous but whose fibers are not parallel to one another. Hence they are not restrictions of a Hopf fibration of $S^7$.
    
      In Section~\ref{Sec:Nonstandard}, we will carefully construct these nonstandard fibrations, and prove that they are indeed distinct from the Hopf fibrations and that they are fiberwise homogeneous. In Section~\ref{Sec:Geometry}, we will elaborate on the geometry of these nonstandard fibrations, and explain how they differ from the Hopf fibrations.
      
    \subsection{Acknowledgments}
      This article is an extended version of a portion of my doctoral dissertation. I am very grateful to my advisor Herman Gluck, without whose encouragement, suggestions, and just overall support and cameraderie, this would never have been written.
      
      Thanks as well to the Math Department at the University of Pennsylvania for their support during my time there as a graduate student.
    
  \section{The Nonstandard fibrations of \texorpdfstring{$S^3 \times S^3$}{S3 x S3}}\label{Sec:Nonstandard}
    \subsection{John Petro's moduli space for fibrations of \texorpdfstring{$S^3 \times S^3$}{S3 x S3}}
      By ``great \mbox{3-sphere},'' we mean a subset of $S^3 \times S^3$ which, when included in $S^7$ as the Clifford Torus, is a great 3-sphere. Equivalently, a great 3-sphere in $S^3 \times S^3$ is the graph of an isometry from one 3-sphere factor to the other.
    
      We identify $S^3$ with the unit sphere in the space of quaternions. This allows us to identify $SO(4)$ with $S^3\times S^3 / \Z_2$, with the pair of unit quaternions $(p,q)$ acting on $x\in S^3$ via quaternionic multiplication:
      \[ (p,q)\cdot x = p\ x\ q^{-1}, \]
      and the $\Z_2$ action taking $(p,q)$ to $(-p,-q)$. Then a typical great 3-sphere in $S^3\times S^3$ can be written as the graph of an orientation-preserving isometry of $S^3$, $x\mapsto p  x  q^{-1}$:
      \[ \Sigma^3 = \{ (x, p\ x\ q^{-1}) : x\in S^3 \}, \]
      or that followed by an orientation-reversing isometry of the second factor.
    
      By the work of John Petro \cite{petro1987great}, inspired by earlier work of Herman Gluck and Frank Warner \cite{gluck1983great}, the fibrations of $S^3\times S^3$ by great 3-spheres are in one-to-one correspondence with four copies of the space of distance-decreasing maps of $S^3$ to $S^3$.
    
      One of those copies consists of fibrations of $S^3\times S^3$ by great 3-spheres $\Sigma^3$ of the form
      \[ \Sigma_p^3 = \{ (x, p\ x\ f(p)^{-1}) : x\in S^3 \}, \]
      where $p$ ranges across $S^3$ and indexes each fiber in the fibration, and $f$ is a fixed distance-decreasing map of $S^3$ to $S^3$ which determines the fibration.

      The other fibrations of $S^3 \times S^3$ by great 3-spheres have similar forms. The key fact which Petro uses to prove this correspondence is that two great 3-spheres
      \begin{align*}
        \Sigma_1 &= \{ (x, p_1\ x\ q_1^{-1} ) : x\in S^3 \} \\
        \Sigma_2 &= \{ (x, p_2\ x\ q_2^{-1} ) : x\in S^3 \}
      \end{align*}
      are disjoint precisely when $d(p_1,p_2) \neq d(q_1,q_2)$. In a fibration, any two fibers are disjoint, and so the inequality becomes either uniformly $>$ for any two fibers, or uniformly $<$ for any two fibers. In the former case we get fibers of the form
      \[ \Sigma_p^3 = \{ (x, p\ x\ f(p)^{-1}) : x\in S^3 \}, \]
      with $f$ a distance-decreasing function, and we get something similar in the case where the inequality is replaced by $<$.
    
    \subsection{The Hopf fibration}
      If we think of the unit 7-sphere as sitting in quaternionic 2-space, then the fibers of the Hopf fibration are the intersections of that 7-sphere with the various quaternionic lines through the origin. To be more precise, quaternionic 2-space has the structure of a ``vector space'' with quaternionic scalars acting on the right and matrices acting on the left. We write the quaternionic line spanned by $(u,v)\in \HH^2$ as $\{(u,v)\cdot x: x\in \HH \}$.
      
      If we restrict the Hopf fibration to $S^3\times S^3$, then we take only the quaternionic lines which intersect this set; that restricts our pairs $(u,v)$ to those where $\|u\|=\|v\|$. We normalize so that $\|u\|=\|v\|=1$. Further, we may choose $u=1$ by right-multiplying by an appropriate scalar. Finally, instead of multiplying the vector $(u,v)$ by any $x\in\HH$, we restrict ourselves to $x\in S^3$. Thus, the Hopf fibers in $S^3\times S^3$ are of the form
      \begin{align*}
        H_v &= \{(1, v) \cdot x : x\in S^3 \} \\
	    &= \{(x, v\ x) : x\in S^3 \},
      \end{align*}
      with $v$ a unit quaternion. To relate this to our discussion of John Petro's moduli space, the Hopf fibration corresponds to the fibration where we take our distance decreasing function to be $f(v) \equiv 1$. We may easily check that if we apply an isometry of $S^3 \times S^3$ to the Hopf fibrations, left- and right-multiplying on either $S^3$ factor, the distance-decreasing functions that arise in John Petro's moduli space are precisely the constant functions. In other words,
      \begin{prop}\label{Prop:HopfConstant}
        A fibration with fibers
        \[ \Sigma_p = \{ (x, p\ x\ f(p)^{-1}) : x\in S^3 \} \]
        is Hopf if and only if $f$ is a constant function.
      \end{prop}
      \begin{rmk}
        Given a fibration with fibers
        \[ \Sigma_p = \{ (x, p\ x\ f(p)^{-1}) : x\in S^3 \}, \]
        it follows that two fibers
        \begin{align*}
          \Sigma_p &= \{ (x, p\ x\ f(p)^{-1}) : x\in S^3 \} \\
          \Sigma_q &= \{ (x, q\ x\ f(q)^{-1}) : x\in S^3 \}
        \end{align*}
        with $p\neq q$ are parallel if and only if $f(p) = f(q)$: if $f(p) = f(q)$, then both fibers belong to a Hopf fibration (where the constant function $f$ takes the value of $f(p)$). The reverse direction will follow later in Corollary~\ref{Cor:Convenient}.
      \end{rmk}
      
      We demonstrate that the Hopf fibration is indeed fiberwise homogeneous. For each $q\in S^3$, let $q$ act isometrically on $S^3\times S^3$ via $q\cdot (x,y) = (x,q\ y)$. Then it's clear that $q\cdot H_v = H_{qv}$, so that $S^3$ acts transitively on the Hopf fibers.
      
      We remark that this is not the largest group acting on the Hopf fibration; all we need is that it's large enough to act transitively taking fibers to fibers.
      
    \subsection{The new nonstandard examples}
      We will find a distance-decreasing map of $S^3$ to $S^3$ which is nonconstant and is pointwise homogeneous:
      \begin{dfn}
        A function $f:S^3 \to S^3$ is {\em pointwise homogeneous} if for all \mbox{$u, v \in S^3$} there are isometries $T_1, T_2 \in SO(4)$ such that $f\circ T_1 = T_2 \circ f$ and $T_1(u) = v$.
      \end{dfn}
      Finding such a map leads immediately to a fiberwise homogeneous fibration:
      \begin{prop}
        Let $f:S^3 \to S^3$ be distance-decreasing, and consider the fibration whose fibers are
        \[ \Sigma_p = \{ (x, p\ x\ f(p)^{-1}) : x\in S^3 \}. \]
        The fibration is fiberwise homogeneous if and only if $f$ is pointwise homogeneous.
      \end{prop}
      \begin{proof}
        Let $f$ be pointwise homogeneous, and let $u,v\in S^3$. The isometries $T_1, T_2$ in $SO(4)$ which commute with $f$ (and for which $T_1(u)=v$) can be written
        \begin{align*}
          T_1(x) &= p_1\ x\ q_1^{-1}, \\
          T_2(x) &= p_2\ x\ q_2^{-1}.
        \end{align*}
        We mix and match $T_1$ and $T_2$: a straightforward computation shows that the isometry of $S^3 \times S^3$ which takes $(x,y)$ to $(q_1\ x\ q_2^{-1}, p_1\ y\ p_2^{-1})$ takes the fiber $\Sigma_u$ to $\Sigma_v$, and preserves fibers.
        
        Conversely, let the fibration be fiberwise homogeneous. The mixing and matching we performed in the one direction works just the same in reverse.
      \end{proof}
      
      The key insight of this paper is that the classical Hopf fibration $\pi_H: S^3 \to S^2$ can be tweaked to give a distance-decreasing map from $S^3$ to $S^3$ which is pointwise homogeneous, and which then, by the above proposition and the work of John Petro, yields a fiberwise homogeneous fibration of $S^3 \times S^3$ by great 3-spheres which is not part of a Hopf fibration of the 7-sphere.
      \\

      Let $S^2$ be the unit sphere in the space of purely imaginary quaternions. Then $\pi_H$ can be written explicitly as
      \[ \pi_H(p) = p\ i\ p^{-1}. \]
      The great circle fibers are the left cosets $pS^1$ of the circle subgroup
      \[ S^1 = \cos\theta + i\sin\theta. \]
      
      This map is distance-doubling in directions orthogonal to the fibers. For our purposes, we need a distance-decreasing map, so we shrink the image 2-sphere. Let $f:S^3\to S^2$ be defined by
      \begin{align*}
        f_{\alpha}(p) &= p(\cos\alpha + i\sin\alpha)p^{-1} \\
             &= \cos\alpha + \pi_H(p)\sin\alpha,
      \end{align*}
      where $\alpha$ is a fixed small angle. In particular, if $\alpha = \pi/6$, then the image 2-sphere has radius $1/2$. In that case $f_{\alpha}$ is distance-preserving in directions orthogonal to the Hopf fibers when we measure distances in the intrinsic metric on $S^2$, but is strictly distance-decreasing when distances in the 2-sphere are measured via shortcuts in $S^3$. Similarly if $\alpha$ is less than $\pi/6$ then $f_{\alpha}$ is likewise distance-decreasing.
      
      We remark that not only do we have $f_{\alpha}(p) = f_{\alpha}(p\ e^{it})$ for all $\alpha, p, t$, we also have $f_{\alpha}(q\ p) = q\ f_{\alpha}(p)\ q^{-1}$ for all $p,q$. In other words, $f_{\alpha}$ is pointwise homogeneous (left-multiplication by $q$ is conjugate to the isometry of conjugation by $q$).
      
      We take as our nonstandard fiberwise homogeneous fibration of $S^3\times S^3$ the collection of fibers
      \[ \Sigma_p^3 = \{ (x, p\ x\ f_{\alpha}(p)^{-1}) : x\in S^3 \}, \]
      for each $p \in S^3$, and for $0 \leq \alpha \leq \pi/6$.
      
      \begin{thm}
        The fibers $\Sigma_p^3$ form a fiberwise homogeneous fibration of $S^3\times S^3$, distinct from any Hopf fibration, as long as $0 < \alpha \leq \pi/6$.
      \end{thm}
      \begin{proof}
        We have already seen that these great 3-spheres form a fibration, from Petro's moduli space and the observation that $f_{\alpha}$ is distance-decreasing. This fibration is not a Hopf fibration because $f_{\alpha}$ is nonconstant for $\alpha\in (0, \pi/6]$. It remains to show that the fibration is fiberwise homogeneous.
        
        Let $q \in S^3$. We let $S^3$ act isometrically on $S^3\times S^3$ via
        \[ q \cdot (x, y) = ( x\ q^{-1}, q\ y\ q^{-1} ). \]
        Then we claim that $q \cdot \Sigma_p^3 = \Sigma_{qp}^3$. We compute:
        \begin{align*}
          q \cdot \Sigma_p^3 &= q \cdot \{ (x, p\ x\ f_{\alpha}(p)^{-1}) : x\in S^3 \} \\
          &= \{ (x\ q^{-1} , q\ p\ x\ f_{\alpha}(p)^{-1}\ q^{-1} ) : x\in S^3 \} \\
          &= \{ (y , q\ p\ y\ q\ f_{\alpha}(p)^{-1}\ q^{-1} ) : y\in S^3 \}\quad\mbox{ (with $y=x\ q^{-1}$)} \\
          &= \{ (y , (q p)\ y\ f_{\alpha}(q p)^{-1} ) : y\in S^3 \}\quad\mbox{(homogeneity of $f_{\alpha}$)} \\
          &= \Sigma_{qp}^3.
        \end{align*}
        In the third equality we merely substitute $y = x\ q^{-1}$, and in the fourth equality we use the pointwise homogeneity of $f_{\alpha}$ which we remarked upon earlier. This concludes the proof that our fibration is fiberwise homogeneous.
      \end{proof}

  \section{Geometry of the nonstandard fibration}\label{Sec:Geometry}
    In each of the classical Hopf fibrations, the fibers are parallel to one another. This allows us to get a feel for the geometry of the fibrations by putting a metric on the base space of the fibrations which makes their projections into Riemannian submersions. Likewise for the restriction of the Hopf fibration of $S^7$ by 3-spheres to the Clifford torus: the fibers are parallel to one another, and the base is diffeomorphic to $S^3$, and the submersion metric is a round metric.
    
    But in our new nonstandard example, the fibers are not parallel to one another, so we cannot put a submersion metric on the base. We have to get a feel for the geometry by other means. Refer again to Figure~\ref{Fig:ExampleFWH} for an example to keep in mind, of a fiberwise homogeneous fibration with nonparallel fibers.
%
        
    Here is what we do: we fix our attention on one fiber $\Sigma$, and we aim for a description of how that fiber sits in relation to nearby fibers. To be more specific, for each nearby fiber $\Sigma'$, we determine the subset of $\Sigma$ which lies closest to $\Sigma'$. If our fibration were Riemannian, there would be no distinguished subset of $\Sigma$, but because the fibration does not have parallel fibers, as our attention flits over various nearby $\Sigma'$, we will notice various ``hot'' subsets of $\Sigma$ closest to $\Sigma'$ and various ``cold'' subsets of $\Sigma$ which are furthest from $\Sigma'$. See Figure~\ref{Fig:HotAndCold}.
    
    \begin{figure}[h]
      \centering
      \includegraphics[width=0.5\textwidth]{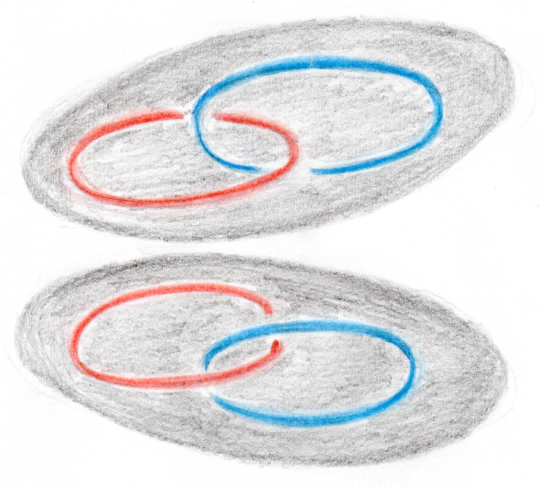}
      \caption{If two disjoint great 3-spheres in $S^3 \times S^3$ are not parallel to one another, then each will have a unique great circle of closest approach to the other, shown in red and called the {\em hot circle}, and a unique great circle of furthest separation from the other, shown in blue and called the {\em cold circle}. The hot circles are parallel to one another; so are the cold circles, but they are further apart.}
      \label{Fig:HotAndCold}
    \end{figure}
            
    Now fix a value of $\alpha$ in $(0,\pi/6]$, and consider our nonstandard fibration with fibers
    \[ \Sigma_p = \{ (x, p\ x\ f_{\alpha}(p)^{-1}) : x\in S^3 \}, \]
    for each $p\in S^3$, and with $f_{\alpha}(p)=p\ e^{i\alpha}\ p^{-1}$. We fix our attention on the fiber
    \[ \Sigma_1 = \{ (x, x\ e^{-i\alpha}) : x\in S^3 \}, \]
    and determine how the nearby fibers are positioned relative to it. By ``nearby fibers,'' we mean those fibers $\Sigma_p$ for which $p$ is on the boundary of an $\varepsilon$-neighborhood of $1\in S^3$.
    
    First we observe that if we move away from $\Sigma_1$ in the $i$-direction, we find that the fibers $\Sigma_{e^{i\varepsilon}}$ and $\Sigma_{e^{-i\varepsilon}}$ are parallel to $\Sigma_1$. That's because
    \[ f_{\alpha}(1) = f_{\alpha}(e^{i\varepsilon}) = f_{\alpha}(e^{-i\varepsilon}) = e^{i\alpha};\]
    the function $f_{\alpha}$ is constant along the Hopf fibers (left cosets $p\ e^{i\theta}$).
    
    Thus moving from fiber to fiber in the $i$-direction, our nonstandard fibration looks just like a Hopf fibration. The new interesting geometry lies in the orthogonal directions, as we let $p$ vary along the circle $p(\theta) = e^{j_{\theta}\varepsilon} = e^{(j\cos\theta+k\sin\theta)\varepsilon}$. Here we let $j_{\theta}$ denote a point on the great circle through $j$ and $k$. We will see momentarily that for those values of $p$, the fiber $\Sigma_1$ has a great circle's worth of closest points to $\Sigma_{p(\theta)}$ and a great circle's worth of furthest points from $\Sigma_{p(\theta)}$. Refer again to Figure~\ref{Fig:HotAndCold}.
        
    We answer the question ``what is this nonstandard fibration's geometry like?'' by describing how these great circles vary with $\theta$. Because our fibration is fiberwise homogeneous, the geometry near $\Sigma_1$ is the same as the geometry near any other fiber.
    
    Our task is now to determine the ``hot'' and ``cold'' sets in $\Sigma_1$ relative to the fibers $\Sigma_{p(\theta)}$; i.e.~the subsets of $\Sigma_1$ closest to and furthest from those nearby fibers.
    
    To compute the hot and cold sets in $\Sigma_1$ relative to $\Sigma_{p(\theta)}$, we take advantage of the fact that they will be preserved by isometries of the ambient space, and apply an isometry of $S^3\times S^3$ (depending on $\theta$) which takes $\Sigma_1$ to the diagonal, and takes $\Sigma_{p(\theta)}$ to a conveniently placed great 3-sphere. The following lemma demonstrates what it means to be a ``conveniently placed great 3-sphere.''
    
    \begin{lem}
      Let $\bigtriangleup$ and $\Sigma$ be the great 3-spheres
      \begin{align*}
        \bigtriangleup &= \{(x, x) : x\in S^3 \}, \\
        \Sigma &= \{(y, e^{i\theta}\ y\ e^{-i\phi}) : y\in S^3 \},
      \end{align*}
      with $0 < \phi < \theta < \pi$. Then the great circle
      \[ \{(\cos t+i\sin t, \cos t+i\sin t) : t\in [0,2\pi]\} \]
      is the set of points in $\bigtriangleup$ closest to $\Sigma$ (the ``hot'' set), and the great circle
      \[ \{(j\cos t+k\sin t, j\cos t+k\sin t) : t\in [0,2\pi]\} \]
      is the set of points in $\bigtriangleup$ furthest from $\Sigma$ (the ``cold'' set).
    \end{lem}
    \begin{proof}
      The ``hot'' set $H$ and ``cold'' set $C$ are the subsets of $\bigtriangleup$ given by
      \begin{align*} 
        H &= \{ (x,x)\in \bigtriangleup : d((x,x),\Sigma) = \inf_{(z,z)\in \bigtriangleup}d((z,z),\Sigma) \} \\
        C &= \{ (x,x)\in \bigtriangleup : d((x,x),\Sigma) = \sup_{(z,z)\in \bigtriangleup}d((z,z),\Sigma) \}
      \end{align*}
      We compute, for $(z,z)\in\bigtriangleup$ (see the discussion afterwards for explanation):
      \begin{align*}
        d((z,z),\Sigma) &= \inf_{y\in S^3} \sqrt{d(z,y)^2 + d(z,e^{i\theta}\ y\ e^{-i\phi})^2} \\
        &= \inf_{y\in S^3} \sqrt{d(z,y)^2 + d(e^{-i\theta}\ z\ e^{i\phi},y)^2} \\
        &= \sqrt{d(z,\overline{z})^2 + d(e^{-i\theta}\ z\ e^{i\phi},\overline{z})^2} \\
        &= \frac{1}{\sqrt{2}} d(z, e^{-i\theta}\ z\ e^{i\phi}).
      \end{align*}
      This requires explanation. We use $d$ to denote both distance in $S^3\times S^3$ and in $S^3$, relying on context to distinguish between the two. The first equality is by the definition of the product metric. To minimize this quantity, we have to pick $y$ to minimize the sum of squares of distances between a fixed point $z$ and two points that vary with $y$, see Figure~\ref{Fig:Infimums}, left picture. The second equality follows because $L_{e^{-i\theta}}\circ R_{e^{i\phi}}$ is an isometry of $S^3$. Now we have the easier task of minimizing the sum of squares of distances between two fixed points and a single point $y$ that varies, see Figure~\ref{Fig:Infimums}, right picture. In the third equality, we minimize the given quantity by choosing $y=\overline{z}$, the midpoint of $z$ and $e^{-i\theta}\ z\ e^{i\phi}$.
      
      \begin{figure}[h]
        \centering
        \includegraphics[width=0.9\textwidth]{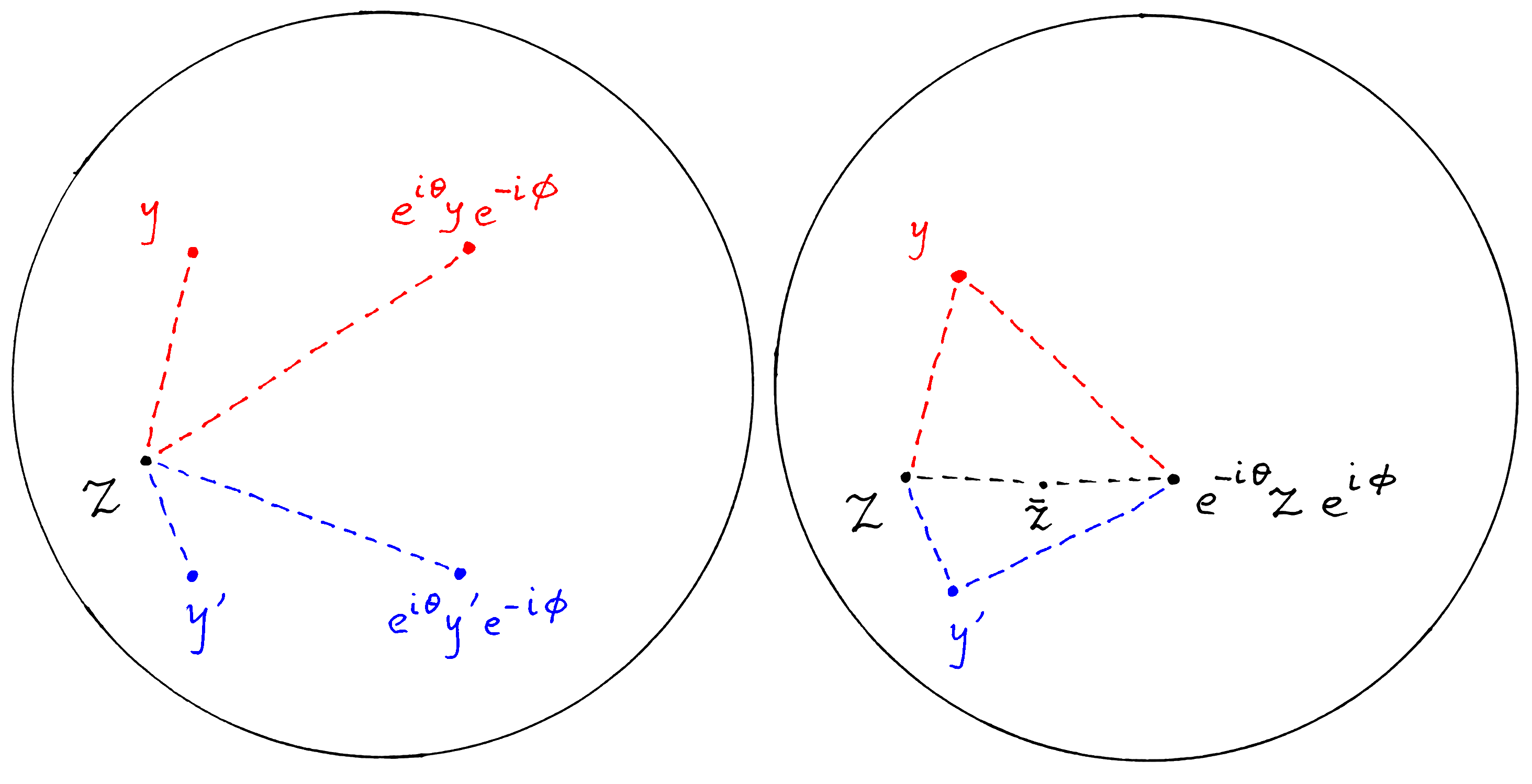}
        \caption{In the left picture, it's difficult to see which choice of $y$ will minimize the sum of the squares of the dashed distances. In the right picture, it's easier to see that the correct choice is the midpoint of $z$ and $e^{-i\theta}ze^{i\phi}$.}
        \label{Fig:Infimums}
      \end{figure}
      
      To find $H$, we find the infimum of $d(z, e^{-i\theta}\ z\ e^{i\phi})$, and see that the values of $z$ which minimize it are precisely those on the great circle joining $1$ and $i$. Similarly, to find $C$, we find the supremum of $d(z, e^{-i\theta}\ z\ e^{i\phi})$, and see that the values of $z$ which maximize it are precisely those on the great circle joining $j$ and $k$.
      
      This takes a small amount of computation to see. We observe that
      \[ d(z, e^{-i\theta}\ z\ e^{i\phi}) = d(z^{-1}\ e^{i\theta}\ z, e^{i\phi}). \]
      As $z$ varies, the point $z^{-1}\ e^{i\theta}\ z$ remains on a sphere centered at $1$. But $e^{i\theta}$ and $e^{i\phi}$ lie on half a great circle with endpoints $\pm 1$, so the values of $z$ which minimize the expression above are precisely those which leave $e^{i\theta}$ where it is instead of moving it further away. Those values of $z$ are exactly the great circle through $1$ and $i$. A similar argument shows that for $z$ on the great circle through $j$ and $k$, the expression $z^{-1}\ e^{i\theta}\ z$ is equal to $e^{-i\theta}$ and hence is maximally far from $e^{i\phi}$. See Figure~\ref{Fig:Minimum}.
      \begin{figure}[h]
        \centering
        \includegraphics[width=0.4\textwidth]{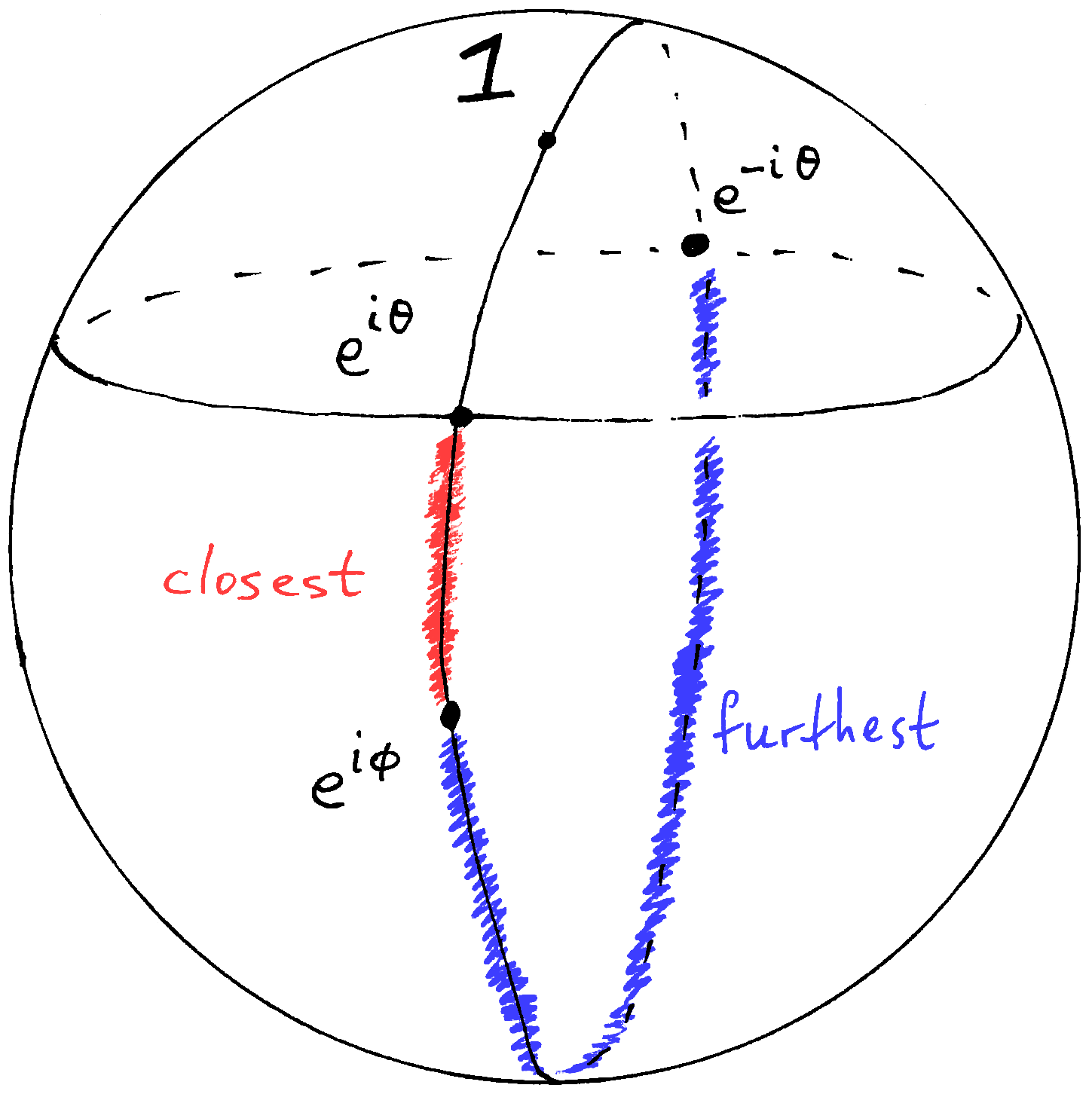}
        \caption{As $z$ varies, $z^{-1}\ e^{i\theta}\ z$ remains on a sphere centered at $1$. For $z$ on the great circle containing $1$ and $i$, $z^{-1}\ e^{i\theta}\ z$ remains at $e^{i\theta}$, close to $e^{i\phi}$. For $z$ on the great circle containing $j$ and $k$, $z^{-1}\ e^{i\theta}\ z$ moves to $e^{-i\theta}$, maximally far from $e^{i\phi}$.}
        \label{Fig:Minimum}
      \end{figure}      
    \end{proof}
    \begin{cor}\label{Cor:Convenient}
      The above lemma also holds if we replace $(i,j,k)$, wherever we see them, with any positive orthonormal basis for the space of purely imaginary quaternions. The proof works exactly the same way. Thus a ``conveniently placed 3-sphere'' relative to the diagonal is one of the form
      \[ \Sigma = \{(y, e^{q\theta}\ y\ e^{-q\phi}) : y\in S^3 \} \]
      with $q$ any purely imaginary unit quaternion.
    \end{cor}
    
    Recall that we have defined
    \begin{align*}
      \Sigma_1 &= \{ (x, xe^{-i\alpha} ) : x\in S^3 \}, \\
      \Sigma_{p(\theta)} &= \{ (y, e^{j_{\theta}\varepsilon}\ y\ f_{\alpha}(e^{j_{\theta}\varepsilon})^{-1} : y\in S^3 \} \\
       &= \{ (y, e^{j_{\theta}\varepsilon}\ y\ e^{j_{\theta}\varepsilon}\ e^{-i\alpha}\ e^{-j_{\theta}\varepsilon}) : y \in S^3 \},
    \end{align*}
    with $j_{\theta} = j\cos\theta + k\sin\theta$. We would like to move $\Sigma_1$ to the diagonal $\bigtriangleup$, and also move $\Sigma_{p(\theta)}$ to a conveniently placed great 3-sphere. To that end, we let \mbox{$T\in SO(4)\times SO(4)$} be the isometry of $S^3\times S^3$ defined by
    \[ T(x,y) = (x, y\ e^{i\alpha}). \]
    Then we compute
    \begin{align*}
      \Sigma_1^{'} &= T\cdot \Sigma_1 = \bigtriangleup \\
      \Sigma_{p(\theta)}^{'} &= T\cdot \Sigma_{p(\theta)} = \{ (y, e^{j_{\theta}\varepsilon}\ y\ e^{j_{\theta}\varepsilon}\ e^{-i\alpha}\ e^{-j_{\theta}\varepsilon}\ e^{i\alpha}) : y\in S^3 \}.
    \end{align*}
    We let $q' = e^{j_{\theta}\varepsilon}\ e^{-i\alpha}\ e^{-j_{\theta}\varepsilon}\ e^{i\alpha}$. We find (omitting the tedious computation) that
    \begin{align*}
      q' &= (\cos^2\varepsilon + \sin^2\varepsilon \cos 2\alpha) + (\sin^2\varepsilon \sin 2\alpha)i \\
	& \quad + (2\cos\varepsilon \sin\varepsilon \sin\alpha)(j\cos(\theta+\alpha-\frac{\pi}{2})+k\sin(\theta+\alpha-\frac{\pi}{2})).
    \end{align*}
    If we let $\varepsilon \to 0$, and take a first-order approximation, then we conclude that
    \[ q' \sim 1 + 2\varepsilon\sin\alpha (j\cos(\theta+\alpha-\frac{\pi}{2})+k\sin(\theta+\alpha-\frac{\pi}{2})). \]
    Therefore $\Sigma_{p(\theta)}^{'}$ is not yet conveniently positioned relative to $\bigtriangleup$, because
    \[ j_{\theta}\neq j\cos(\theta+\alpha-\frac{\pi}{2})+k\sin(\theta+\alpha-\frac{\pi}{2}). \]
    
    Therefore we let $T'\in SO(4)\times SO(4)$ act on $S^3\times S^3$ via
    \[ T'(x,y) = (x e^{i\frac{1}{2}(\frac{\pi}{2}-\alpha)}, y e^{i\frac{1}{2}(\frac{\pi}{2}-\alpha)} ). \]
    This choice is exactly what we need to move $\Sigma_{p(\theta)}^{'}$ to a convenient position. We find that
    \begin{align*}
      \Sigma_1^{''} &= T'\cdot \Sigma_1^{'} = \bigtriangleup \\
      \Sigma_{p(\theta)}^{''} &= T'\cdot \Sigma_{p(\theta)}^{'} = \{ (y, e^{j_{\theta}\varepsilon}\ y\ e^{-i\frac{1}{2}(\frac{\pi}{2}-\alpha)}\ q'\ e^{i\frac{1}{2}(\frac{\pi}{2}-\alpha)}) : y\in S^3 \}.
    \end{align*}
    Computing $q'' = e^{-i\frac{1}{2}(\frac{\pi}{2}-\alpha)}\ q'\ e^{i\frac{1}{2}(\frac{\pi}{2}-\alpha)}$, we find (omitting, again, the tedious computation) that
    \begin{align*}
      q'' &= (\cos^2\varepsilon + \sin^2\varepsilon \cos 2\alpha) + (\sin^2\varepsilon \sin 2\alpha)i \\
	& \quad + (2\cos\varepsilon \sin\varepsilon \sin\alpha)(j\cos(\theta)+k\sin(\theta)).
    \end{align*}
    Letting $\varepsilon \to 0$, and taking a first-order approximation, we see that
    \[ q'' \sim 1 + 2\varepsilon\sin\alpha j_{\theta}.\]
    Therefore $\Sigma_{p(\theta)}^{''}$ is conveniently placed relative to $\bigtriangleup$ (see Corollary~\ref{Cor:Convenient}), at least in the limit as $\varepsilon \to 0$.
    
    It follows that the ``hot'' circle in $\bigtriangleup$ relative to $\Sigma_{p(\theta)}^{''}$ is the great circle passing through $1$ and $j_{\theta}$, and the ``cold'' circle is the orthogonal great circle passing through $i$ and $j_{\theta+\pi/2}$, at least as $\varepsilon \to 0$. All that remains is to move $\bigtriangleup$ back to $\Sigma_1$ and see where that takes these hot and cold circles.
    
    \begin{thm}
      The set of points in $\Sigma_1$ which lies closest to $\Sigma_{p(\theta)}$ is the great circle which, when projected to the first factor of $S^3\times S^3$, passes through $e^{i(\frac{\alpha}{2}-\frac{\pi}{4})}$ and $(j\cos\theta + k\sin\theta)e^{i(\frac{\alpha}{2}-\frac{\pi}{4})}$.
      
      The set of points in $\Sigma_1$ lying furthest from $\Sigma_{p(\theta)}$ is the great circle which, when projected to the first factor, passes through $i\ e^{i(\frac{\alpha}{2}-\frac{\pi}{4})}$ and $(-j\sin\theta + k\cos\theta)e^{i(\frac{\alpha}{2}-\frac{\pi}{4})}$.
    \end{thm}
    \begin{proof}
      We simply apply Corollary~\ref{Cor:Convenient} to $\bigtriangleup$ and $\Sigma_{p(\theta)}^{''}$, and follow what happens to the hot and cold circles in $\bigtriangleup$ as we undo the transformations $T'$ and $T$.
    \end{proof}
    As a final remark, we observe that the hot circles in $\Sigma_1$ always pass through the antipodal points $\pm e^{i(\frac{\alpha}{2}-\frac{\pi}{4})}$ independently of $\theta$. Likewise the cold circles always pass through $\pm i\ e^{i(\frac{\alpha}{2}-\frac{\pi}{4})}$ independently of $\theta$. As we vary $\theta$, we watch the hot and cold circles spin around their fixed points to trace out 2-spheres. The hot and cold circles spin round one another like a pair of linked eggbeater blades. See Figure~\ref{Fig:eggbeater}. In the center we have $\Sigma_1$, with the circle of fibers $\Sigma_{e^{\varepsilon(j\cos\theta+k\sin\theta)}}$ arrayed around it. Inside $\Sigma_1$, the linked red and blue great circles are the hot and cold circles.
    \begin{figure}[h]
      \centering
      \includegraphics[width=0.5\textwidth]{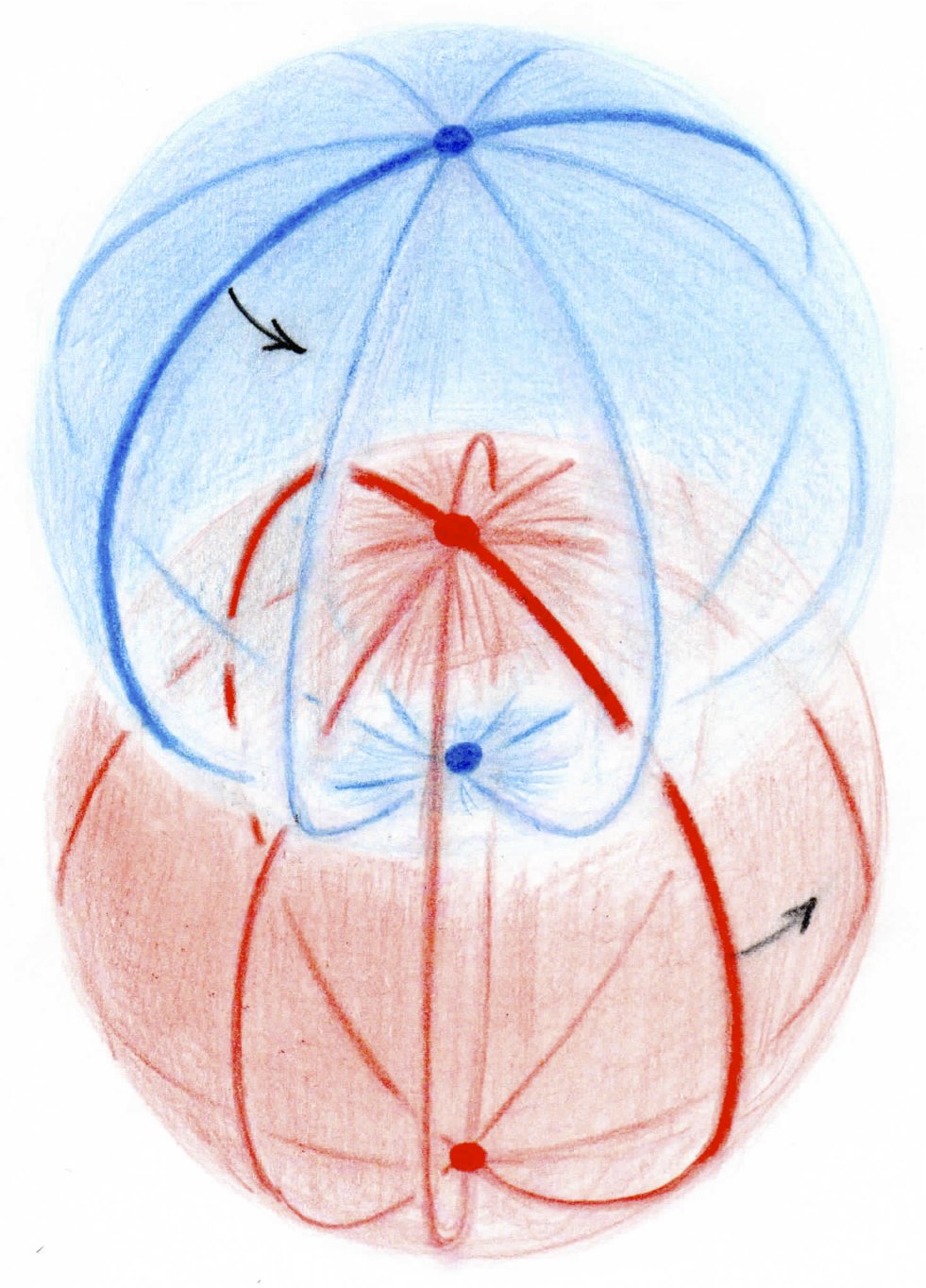}
      \caption{In our unusual fiberwise homogeneous fibration of \mbox{$S^3\times S^3$} by great 3-spheres, we progress along a circle's worth of fibers about a given one but not parallel to it, and record on the given fiber the corresponding progression of hot and cold circles. These remind us of the rotating blades of an eggbeater, with all the hot circles pivoting around a common pair of antipodal points, and likewise for the cold circles. As in a real eggbeater, collision of the blades is avoided by their coordinated progression.}
      \label{Fig:eggbeater}
    \end{figure}

\nocite{wong1961isoclinic, wolf1963elliptic, wolf1963geodesic, escobales1975riemannian, gromoll1985one, ranjan1985riemannian, gluck1986geometry, gluck1987fibrations}

\bibliography{thesis}{}

\providecommand{\bysame}{\leavevmode\hbox to3em{\hrulefill}\thinspace}
\providecommand{\MR}{\relax\ifhmode\unskip\space\fi MR }
\providecommand{\MRhref}[2]{%
  \href{http://www.ams.org/mathscinet-getitem?mr=#1}{#2}
}
\providecommand{\href}[2]{#2}
\begin{thebibliography}{10}

\bibitem{escobales1975riemannian}
Richard~H Escobales~Jr, \emph{Riemannian submersions with totally geodesic
  fibers}, Journal of Differential Geometry \textbf{10} (1975), no.~2,
  253--276.

\bibitem{gluck1983great}
Herman Gluck and Frank Warner, \emph{Great circle fibrations of the
  three-sphere}, Duke Mathematical Journal \textbf{50} (1983), no.~1, 107--132.

\bibitem{gluck1986geometry}
Herman Gluck, Frank Warner, and Wolfgang Ziller, \emph{The geometry of the
  {H}opf fibrations}, Enseign. Math.(2) \textbf{32} (1986), no.~3-4, 173--198.

\bibitem{gluck1987fibrations}
\bysame, \emph{Fibrations of spheres by parallel great spheres and {B}erger's
  rigidity theorem}, Annals of Global Analysis and Geometry \textbf{5} (1987),
  no.~1, 53--82.

\bibitem{gromoll1985one}
Detlef Gromoll and Karsten Grove, \emph{One-dimensional metric foliations in
  constant curvature spaces}, Differential geometry and complex analysis,
  Springer, 1985, pp.~165--168.

\bibitem{gromoll1988low}
\bysame, \emph{The low-dimensional metric foliations of {E}uclidean spheres},
  J. Differential Geom \textbf{28} (1988), no.~1, 143--156.

\bibitem{hopf1931abbildungen}
Heinz Hopf, \emph{{\"U}ber die {A}bbildungen der dreidimensionalen {S}ph{\"a}re
  auf die {K}ugelfl{\"a}che}, Mathematische Annalen \textbf{104} (1931), no.~1,
  637--665.

\bibitem{hopf1935abbildungen}
\bysame, \emph{{\"U}ber die {A}bbildungen von {S}ph{\"a}ren auf {S}ph{\"a}re
  niedrigerer {D}imension}, Fundamenta Mathematicae \textbf{25} (1935), no.~1,
  427--440.

\bibitem{nuchi2014space}
Haggai Nuchi, \emph{Fiberwise homogeneous fibrations of the 3-dimensional space
  forms by geodesics}, \href{http://arxiv.org/abs/1407.4550}{arXiv:1407.4550}
  (2014).

\bibitem{nuchi2014hopf}
\bysame, \emph{{H}opf fibrations are characterized by being fiberwise
  homogeneous}, \href{http://arxiv.org/abs/1407.4549}{arXiv:1407.4549} (2014).

\bibitem{petro1987great}
John Petro, \emph{Great sphere fibrations of manifolds}, Rocky Mountain Journal
  of Mathematics \textbf{17} (1987), no.~4, 865--886.

\bibitem{ranjan1985riemannian}
Akhil Ranjan, \emph{{R}iemannian submersions of spheres with totally geodesic
  fibres}, Osaka Journal of Mathematics \textbf{22} (1985), no.~2, 243--260.

\bibitem{wilking2001index}
Burkhard Wilking, \emph{Index parity of closed geodesics and rigidity of {H}opf
  fibrations}, Inventiones mathematicae \textbf{144} (2001), no.~2, 281--295.

\bibitem{wolf1963elliptic}
Joseph~A Wolf, \emph{Elliptic spaces in {G}rassmann manifolds}, Illinois J.
  Math \textbf{7} (1963), 447--462.

\bibitem{wolf1963geodesic}
\bysame, \emph{Geodesic spheres in {G}rassmann manifolds}, Illinois J. Math
  \textbf{7} (1963), 425--446.

\bibitem{wong1961isoclinic}
Yung-Chow Wong, \emph{Isoclinic $n$-planes in {E}uclidean $2n$-space,
  {C}lifford parallels in elliptic $(2n-1)$-space, and the {H}urwitz matrix
  equations}, no.~41, American mathematical society, 1961.

\end{thebibliography}
\bibliographystyle{amsplain}
\end{document}